\documentclass{amsart}

\usepackage{graphicx}
\usepackage{graphics}
\usepackage{amssymb, amsmath, amsthm}
\usepackage{geometry}
\usepackage{lipsum}
\usepackage{enumerate}
\usepackage{dsfont}
\usepackage{mathrsfs}
\usepackage{xcolor}
\usepackage{multicol}
\usepackage{dblfloatfix}
\input xy
\xyoption{all}

\usepackage{hyperref}

\newtheorem{theorem}{Theorem}[section]
\newtheorem{proposition}[theorem]{Proposition}
\newtheorem{lemma}[theorem]{Lemma}
\newtheorem{corollary}[theorem]{Corollary}

\theoremstyle{definition}
\newtheorem{definition}[theorem]{Definition}
\newtheorem{example}[theorem]{Example}
\newtheorem{remark}[theorem]{Remark}

\newcommand{\on}{\operatorname}
\newcommand{\GL}{\operatorname{\bf GL}}
\newcommand{\Aut}{\operatorname{\bf Aut}}
\newcommand{\G}{\operatorname{\bf G}}
\renewcommand{\H}{\operatorname{\bf H}}

\newcommand{\mb}{\mathbb}
\newcommand{\mc}{\mathcal}
\newcommand{\ms}{\mathscr}
\newcommand{\m}{\mathfrak{m}}
\newcommand{\Spec}{\operatorname{Spec}}
\newcommand{\x}{\mathbf{x}}

\newcommand{\e}{\mathbf{e}}
\newcommand{\0}{\mathbf{0}}
\newcommand{\g}{\mathfrak{g}}
\newcommand{\h}{\mathfrak{h}}

\newcommand{\p}{\mathfrak{p}}
\newcommand{\q}{\mathfrak{q}}
\newcommand{\X}{\mathcal{X}}
\newcommand{\Y}{\mathcal{Y}}

\begin{document}

\title{Differentiably simple rings and ring extensions defined by $p$-basis} 
\author{Celia del Buey de Andr\'es}
\address{Dpto de Matem\'aticas, Universidad Aut\'onoma de Madrid, Ciudad Universitaria de Cantoblanco, 28049 Madrid, Spain}
\email{celia.delbuey@uam.es}

\author{Diego Sulca}
\address{CIEM-FAMAF, Universidad Nacional de C\'ordoba, Ciudad Universitaria, C\'ordoba X5000HUA, Argentina}
\email{diego.a.sulca@unc.edu.ar}

\author{Orlando E. Villamayor}
\address{Dpto de Matem\'aticas, Universidad Aut\'onoma de Madrid and Instituto de Ciencias Matem\'aticas CSIC-UAM-UC3M-UCM, Ciudad Universitaria de Cantoblanco, 28049 Madrid, Spain}
\email{villamayor@uam.es}


\begin{abstract}
We review the concept of differentiably simple ring and we give a new proof of Harper's Theorem on the characterization of Noetherian differentiably simple rings in positive characteristic. We then study flat families of differentiably simple rings, or equivalently, finite flat extensions of rings which locally admit $p$-basis. These extensions are called {\em Galois extensions of exponent one}. For such an extension $A\subset C$, we introduce an $A$-scheme, called the {\em Yuan scheme}, which parametrizes subextensions $A\subset B\subset C$ such that $B\subset C$ is Galois of a fixed rank. So, roughly, the Yuan scheme can be thought of as a kind of Grassmannian of Galois subextensions. We finally prove that the Yuan scheme is smooth and compute the dimension of the fibers. 
\end{abstract}

\maketitle

\section{Introduction}

Throughout, by a ring we mean a commutative ring with $1$ and $p$ is a prime number.
 
A ring $R$ is called {\em differentiably simple} if the only ideals $I\subset R$ such that $D(I)\subset I$ for any derivation $D\colon R\to R$ are the ring itself and the zero ideal.
In characteristic $p$, the structure of Noetherian differentiably simple rings is well-understood: {\em a Noetherian ring $R$ is differentiably simple if and only if  it is of the form $k[X_1,\ldots,X_n]/\langle X_1^p,\ldots,X_n^p\rangle$, where $k$ is a field}.

The above characterization, which is known as Harper's theorem, was first obtained by Harper under the assumption that $R$ is a finitely generated algebra over an algebraically closed field \cite{Harper61}, and was later completed by Yuan in \cite{Yuan64}.
The original proofs of Harper and Yuan involve somewhat complicated computations with derivations. Shorter proofs were given more recently by Maloo \cite{Maloo93} and Bavula \cite{Bavula08}. 
In Section \ref{section: diff simple rings} of these notes, we present another proof of Harper's theorem which is more conceptual (Theorem \ref{Yuan-Harper theorem}). Next,
in addition to reviewing other properties of differentiably simple rings, we present an alternative proof of the following result by Andr\'e \cite[Proposition 67]{Andre91} that describes the fiber of flat homomorphisms between Noetherian differentiably simple rings: {\em if $(S,\m_S)\subset (R,\m_R)$ is a flat extension of Noetherian differentiably simple rings and $R^p\subset S$, then the fiber $R/\m_S R$ is differentiably simple} (Theorem \ref{th:flat homomorphism of diff simple rings}).

In Section \ref{section: flat families}, we consider flat families of differentiably simple rings. 
This goes back to Yuan \cite{Yuan70}, who introduced a notion of Galois extension for rings to generalize  Jacobson's Galois theory for purely inseparable field extensions of exponent one.  
A ring extension $A\subset C$ is called a {\em Galois extension of exponent one} if
\begin{enumerate}
	\item $C^p\subset A$,
	\item $C$ is a finite projective $A$-module, and
	\item $C$ has locally $p$-basis over $A$. 
\end{enumerate}
The third condition can actually be replaced by the condition that $C\otimes_A \kappa(\p)$ is differentiably simple for each $\p\in\Spec(A)$. 
So, roughly speaking, a finite extension $A\subset C$, such that $C^p\subset A$, is a Galois extension if and only if $C$ is flat over $A$ and all fibers are differentiably simple. 
 
While Yuan proceeds to establish a Galois correspondence between the $A$-subalgebras $B\subset C$ such that $C$ is Galois over $B$ and certain restricted Lie subalgebras of $\on{Der}_A(C)$, we focus on the problem of parameterizing these subalgebras with a geometric object.  
To this end, we introduce in Section \ref{section: Yuan scheme} a functor $\mc{Y}_{C/A}^r$ from the category of $A$-algebras to the category of sets, and then we prove that it is representable by a scheme over $A$, which we call  the {\em Yuan scheme}. By construction, the set of $A$-points of this scheme is in bijection with the set of intermediate rings $B$,  $A\subset B\subset C$, such that $C$ is Galois over $B$ of rank $p^r$. 
We shall prove that the scheme $\mc{Y}_{C/A}^r$ is a locally closed subscheme of a Grassmannian of the $A$-module $C$ (Theorem \ref{th: existence of the Yuan scheme}). Furthemore, we shall show that $\mc{Y}_{C/A}^r$ is smooth over $A$ and we compute the dimension of the fibers (Theorem \ref{th:smoothness over a ring}).

A particular object of interest is the Yuan scheme $\mc{Y}^r_{L/k}$, where $k\subset L$ is a finite purely inseparable extension of fields such that $L^p\subset k$ and $L$ admits a $p$-basis over $k$ of $n$ elements. For each $0<r< n$, we show that $\mc{Y}_{L/k}^r$ is a quasi-projective smooth $k$-variety of dimension $(p^n-p^{n-r})(n-r)$ whose rational points parametrize the subextensions $k\subset K\subset L$ such that $K\subset L$ has a $p$-basis with $r$ elements (see Theorem \ref{th:smoothness over a field}).

\subsection{Notation and conventions}
Throughout, $p$ denotes a fixed rational prime. 
All rings are assumed to be commutative with identity.

If $k$ is a ring and $A,B$ are $k$-algebras, then $\on{Hom}_k(A,B)$ denotes the set of $k$-algebra homomorphisms from $A$ to $B$.
 
If $A$ is a ring and $M$ is an $A$-module, we denote by $\on{Der}(A,M)$ the $A$-module of derivations $D\colon A\to M$. In addition, if $A$ has a structure of a $k$-algebra for some other ring $k$, then $\on{Der}_k(A,M)$ denotes the $A$-submodule of $k$-linear derivations. We write $\on{Der}(A):=\on{Der}(A,A)$ and $\on{Der}_k(A):=\on{Der}_k(A,A)$. 

If $A$ is a local ring, we denote by $\m_A$ and $\kappa(\m)$ its maximal ideal and its residue field respectively.

A scheme $X$ over a ring $k$ is though of interchangeably as a scheme or as a functor from the category of $k$-algebras to the category of sets.

\section{Differentiably simple rings of characteristic $p$}\label{section: diff simple rings}

In this section, we review some basic facts about differentiably simple rings.
We present new proofs for the following two important results: 
the characterization of Noetherian differentiably simple rings (Theorem \ref{Yuan-Harper theorem}), and a theorem describing the fiber of flat homomorphisms between Noetherian differentiably simple rings (Theorem \ref{th:flat homomorphism of diff simple rings}).

\begin{definition}
	Let $R$ be a ring. An ideal $I\subset R$ is \textit{differential} if $D(I)\subset I$ for any derivation $D\colon R\to R$. The ring $R$ is \textit{differentiably simple} if the only differential ideals are $R$ and $0$.
\end{definition}

In positive characteristic $p$, any ideal of $R^p$ extends to a differential ideal in $R$. So if $R$ is differentiably simple, $R^p$ must be a field, and hence $R$ must be local with maximal ideal $\{x\in R: x^p=0\}$. Furthermore, it is easy to prove that any maximal subfield of $R$ that includes $R^p$ is a coefficient field. We summarize this information as follows:

\begin{lemma}[{\cite[Theorem 2.1]{Yuan64}}]\label{Yuan's lemma}
	Let $R$ be a differentiably simple ring of characteristic $p$. Then 
	\begin{enumerate}[(i)]
		\item  $R$ is local and every element $x$ in the maximal ideal satisfies $x^p=0$.
		\item  $R^p$ {is a field} and every maximal subfield of $R$ including $R^p$ is a coefficient field.
	\end{enumerate}
\end{lemma}

\begin{example}\label{ex: example of diff simple}
	Let $L$ be a field of characteristic $p$ and $\{X_\lambda\}_{\lambda\in \Lambda}$ a collection of variables. Then 
	$$R:=L[X_\lambda: \lambda \in \Lambda]/\langle X_\lambda^p: \lambda\in \Lambda\rangle$$ 
	is a differentiably simple ring.
	To prove this we first make two simple observations:
	\begin{itemize}
		\item Since the partial derivatives $\frac{\partial}{\partial X_\lambda}$ of $L[X_\lambda: \lambda\in \Lambda]$ clearly preserve the ideal $\langle X_\lambda^p: \lambda\in \Lambda\rangle$, they induce derivations $\delta_\lambda \colon R\to R$.
		\item The class in $R$ of a polynomial $f\in L[X_\lambda: \lambda\in \Lambda]$ is a unit if and only if the constant term of $f$ is non-zero.
	\end{itemize}
	With these observations, the proof that $R$ is differentiably simple is reduced to showing the following: 
	\begin{enumerate}
		\item[(*)]{\em If $f\in L[X_\lambda: \lambda\in \Lambda]$ is a polynomial included in the ideal $\langle X_\lambda : \lambda\in \Lambda\rangle$, but not in the ideal $\langle X_\lambda^p: \lambda\in \Lambda\rangle$, then some composition of partial derivatives transforms $f$ into a polynomial with non-zero constant term.}
	\end{enumerate} 
	To show this, observe that the condition $f\notin \langle X_\lambda^p: \lambda\in \Lambda\rangle$ amounts to saying that $f$ has a non-zero term $c X_{\lambda_1}^{\beta_1}\cdots X_{\lambda_n}^{\beta_n}$ with $0\leq \beta_i<p$ for $i=1,\ldots,n$.  Let
	\begin{align*}
		D:=\frac{\partial^{\beta_1}}{\partial X_{\lambda_1}^{\beta_1}}\circ\cdots \circ\frac{\partial^{\beta_n}}{\partial X_{\lambda_n}^{\beta_n}}\,.
	\end{align*}
	It is straightforward to check that $D(f)$ has constant term $\beta_1!\cdots\beta_n! \cdot c\neq 0$. This completes the proof of (*).
\end{example}

We now state Harper's characterization of Noetherian differentiably simple rings in characteristic $p$. Observe that the statement makes use of Lemma \ref{Yuan's lemma}.

\begin{theorem}[Harper-Yuan]\label{Yuan-Harper theorem}
	Let $R$ be a Noetherian differentiably simple ring of characteristic $p$, $\{x_1,\ldots,x_n\}$ a minimal generating set for the maximal ideal and $L\subset R$ a coefficient field.
	Then the $L$-algebra homomorphism $L[X_1,\ldots,X_n]\to R$ given by $X_i\mapsto x_i$ induces an isomorphism 
	$L[X_1,\ldots,X_n]/\langle X_1^p,\ldots,X_n^p\rangle\cong R$.
\end{theorem}

Our proof will make use of the next lemma.

\begin{lemma}\label{lifting of derivations lemma}
	Let $S$ be a ring such that the $S$-module of K\"ahler differentials $\Omega_S$ is projective (e.g. a polynomial ring over a field). Let $R=S/I$ be a quotient ring, and let $\pi\colon S\to R$ be the quotient map. 
	If $\delta$ is a derivation of $R$, then there exists a derivation $D\colon S\to S$ such that $\pi\circ D=\delta\circ\pi$.
\end{lemma}

\begin{proof}
	The composition $\delta\circ \pi \colon S\to R$ is a derivation, hence it can be factorized as 
	\begin{align*}
		S\xrightarrow{d} \Omega_{S}\xrightarrow{\varphi} R \,,
	\end{align*}
	where $d$ is the (absolute) universal derivation of $S$, and $\varphi$ is some $S$-module homomorphism. 
	Since $\Omega_S$ is projective, there exists an $S$-module homomorphism $\Phi \colon \Omega_{S}\to S$ such that $\pi\circ\Phi=\varphi$. Then $D:=\Phi\circ d$ is a derivation of $S$ and it satisfies $\pi\circ D=\pi\circ\Phi \circ d=\varphi\circ d=\delta\circ\pi$.
\end{proof}

\medskip

\noindent{\em Proof of Theorem \ref{Yuan-Harper theorem}:}
	Let $\pi\colon L[X_1,\ldots,X_n]\to R$ be the map of the theorem's statement.
	By Lemma \ref{Yuan's lemma}, this map is surjective and its kernel $I:=\ker(\pi)$ includes $\langle X_1^p,\ldots,X_n^p\rangle$. 
	Thus, it only remains to prove that $I$ is exactly $\langle X_1^p,\ldots,X_n^p\rangle$ in order to obtain the desired isomorphism. 
	
	Observe that the assumption that $\{x_1,\ldots,x_n\}$ is a minimal generating set for the maximal ideal of $R$ translates to the inclusion $I\subset \langle X_1,\ldots,X_n\rangle^2$.
	Let $J$ be the ideal generated by $I$ and all the evaluations $D(f)$ for $f\in I$ and $D\in\on{Der}(L[X_1,\ldots,X_n])$. Since $I\subset\langle X_1,\ldots,X_n\rangle^2$, an application of the Leibnitz's rule shows that $J\subset \langle X_1,\ldots,X_n\rangle$. Then it follows that $\pi(J)$ is a proper ideal of $R$. 
	
	We claim that  $\pi(J)$ is a differential ideal of $R$. We have to show that $\delta(\pi(J))\subset\pi(J)$ for any derivation $\delta\colon R\to R$. Given $\delta\in\on{Der}(R)$, Lemma \ref{lifting of derivations lemma} with $S=L[X_1,\ldots,X_n]$ tells us that there is a derivation $D\colon L[X_1,\ldots,X_n]\to L[X_1,\ldots,X_n]$ such that $\pi\circ D=\delta\circ\pi$. Then it is enough to show that $D(J)\subset J$. The equality $\pi\circ D=\delta\circ\pi$ implies that $D(I)\subset I$, hence $D(I)\subset J$ by definition of $J$.
	 Take now any derivation  $E$ of $L[X_1,\ldots,X_n]$ and $f\in I$. We have $D(E(f))=E(D(f))-[E,D](f)$, and this element belongs to $J$ since $D(f)$ and  $f$ belong to $I$, and $E$ and $[E,D]$ belong to $\on{Der}(L[X_1,\ldots,X_n])$. 
	Summarizing, we have proved that $D(g)\in J$ for all $g$ in a generating set for $J$.
	Then an application of the Leibnitz's rule implies that $D(J)\subset J$, as we wanted to show.
	 
	Finally, as $R$ is assumed to be differentiably simple and $\pi(J)$ is a proper differential ideal, we have $\pi(J)=0$, that is $I=J$. Then, an application of the assertion (*) in Example \ref{ex: example of diff simple} shows that necessarily $I\subset \langle X_1^p,\ldots,X_n^p\rangle$. 
	Thus, we conclude that $I=\langle X_1^p,\ldots,X_n^p\rangle$ and the proof of the theorem is complete.
\qed

\medskip

The next corollary will be used in the proof of Theorem \ref{th:flat homomorphism of diff simple rings}.

\begin{corollary}\label{cor: annihilator}
	Let $(R,\m)$ be a Noetherian differentiably simple ring of characteristic $p$ and let $f\in \m\setminus\m^2$. Then $\on{ann}_R(f)=f^{p-1} R$.
\end{corollary}

\begin{proof}
	By Theorem \ref{Yuan-Harper theorem}, we may assume that $R=L[X_1,\ldots,X_n]/\langle X_1^p,\ldots,X_n^p\rangle$ for some field $L$ and that $f$ is the class of $X_1$. Let $R'=L[X_1]/\langle X_1^p\rangle\subset R$. Note that $\on{ann}_{R'}(f)=f^{p-1} R'$. As $R'\subset R$ is clearly flat, we also have $\on{ann}_R(f)=f^{p-1} R$.
\end{proof}

The following proposition describes which quotients of a Noetherian differentiably simple ring are differentiably simple too.

\begin{proposition}\label{prop: quotient of differentiably simple rings}
	Let $(R,\m)$ be a Noetherian differentiably simple ring of characteristic $p$, and let $I$ be a proper ideal of $R$. Then $R/I$ is differentiably simple if and only if $I$ is generated by part of a minimal system of generators of $\m$.
\end{proposition}

\begin{proof}
	Fix a coefficient field $L$ for $R$. If $I$ is generated by part of a minimal generating set for $\m$, then by Theorem \ref{Yuan-Harper theorem}, we may assume that $R=L[X_1,\ldots,X_n]/\langle X_1^p,\ldots,X_n^p\rangle$ and that $I$ is the ideal generated by the classes of $X_1,\ldots,X_r$ for some $r\leq n$. Then $R/I$ is isomorphic to $L$ if $r=n$ or  $L[X_{r+1},\ldots,X_n]/\langle X_{r+1}^p,\ldots,X_n^p\rangle$ if $r<n$. Hence, $R/I$ is differentiably simple, as we have shown in Example \ref{ex: example of diff simple}.
	
	Conversely, assume that $R/I$ is differentiably simple.	Take $x_1,\ldots,x_r\in I$, whose classes modulo $\m^2$ form a basis for the $L$-vector space $(I+\m^2)/\m^2$, and $x_{r+1},\ldots,x_n\in\m$ whose classes modulo $I+\m^2$ form a basis for the $L$-vector space $\m/(I+\m^2)$. 
	Then $\{x_1,\ldots,x_n\}$ is a minimal generating set for $\m$, and $x_{r+1},\ldots,x_n$ induce a minimal generating set for $\m/I\subset R/I$.
	Our suitable choice of the generating system enables us to assert that the following diagram of homomorphisms of $L$-algebras
	\begin{align*}
		\xymatrix{L[X_1,\ldots,X_n]\ar[r]^{\varphi}\ar[d]^{\pi}& L[X_{r+1},\ldots,X_n]\ar[d]^{\bar{\pi}}\\ R\ar[r]^{\bar{\varphi}} &R/I}
	\end{align*}
	where $\pi(X_i)=x_i$ for $i=1,\ldots, n$, $\bar{\pi}(X_i)=x_i+I$ for $i=r+1,\ldots,n$, and ${\varphi}$ and $\bar{\varphi}$ are the obvious quotient maps, is commutative.
	Since $R/I$ is differentiably simple, we have 
	$\ker(\bar{\pi})=\langle X_{r+1}^p,\ldots,X_n^p\rangle$ by Theorem \ref{Yuan-Harper theorem}. Therefore, by commutativity of the diagram, we have $\pi^{-1}(I)=\varphi^{-1}(\ker(\bar{\pi}))=\langle X_1,\ldots,X_r,X_{r+1}^p,\ldots,X_n^p\rangle$. Then it follows that $I=\langle x_1,\ldots,x_r\rangle$. In particular, $I$ is generated by part of a minimal generating set for $\m$.
\end{proof}

\begin{remark}
	Proposition \ref{prop: quotient of differentiably simple rings} follows also from the main result of \cite{Maloo93}, which describes the maximal differential ideals of a Noetherian local ring.
\end{remark}

\begin{proposition}\label{subring of a differentiably simple ring}
	Let $R$ be a differentiably simple ring, not necessarily of characteristic $p$, and $S\subset R$ a subring such that $R$ is a free $S$-module. Then $S$ is also differentiably simple.
\end{proposition}

\begin{proof}
	Let $I$ be a differential ideal of $S$. We have to show that either $I=S$ or $I=0$.
	We first claim that the extended ideal $IR$ is a differential ideal of $R$. 
	In fact, let $D\colon R\to R$ be any derivation and $\{r_{\lambda}:\lambda\in \Lambda\}$ a basis for $R$ as an $S$-module. Then there are additive maps $D_{\lambda} \colon S\to S$ such that $D(s)=\sum_{\lambda\in \Lambda} D_{\lambda}(s)r_{\lambda}$ for all $s\in S$. One easily checks that each $D_{\lambda}$ is a derivation of $S$. As $I$ is a differential ideal of $S$, $D_{\lambda}(I)\subset I$ for all $\lambda\in \Lambda$ and it follows that $D(I)\subset IR$. Therefore, $D(IR)\subset IR$. As $D$ is an arbitrary derivation of $R$, the claim is proved.
	Finally, given that $R$ is differentiably simple, either $IR=R$ or $IR=0$, and since $R$ is a free $S$-module, we conclude that either $I=S$ or $I=0$, as we wanted to show.
\end{proof}

The argument in  the above proof appeared already in the proof of \cite[Theorem 11]{Yuan64} and was also used in \cite{Andre91} to prove one part of the following theorem.

\begin{theorem}[{\cite[Proposition 67]{Andre91}}]\label{th:flat homomorphism of diff simple rings}
	Let $(S,\m_{S})\to (R,\m_R)$ be a flat local homomorphism of Noetherian local rings such that $R^p\subset S$. If $R$ is differentiably simple, then so are $S$ and $R/\m_{S} R$.  
\end{theorem}

\begin{proof} 
	Since $R$ is differentiably simple, we have in particular that $x^p=0$ for any $x\in\m_S$.
	Thus, $S$ is Artinian, and flatness of $R$ over $S$ implies now that $R$ is a free $S$-module.
	Therefore, we can use Proposition \ref{subring of a differentiably simple ring} to deduce that $S$ is differentiably simple. 

	Now, we prove that $R/\m_S R$ is differentiably simple by induction on the embedded dimension of $S$, $\on{embdim}(S)$.
	If $\on{embdim}(S)=0$, then $\m_{S} R=0$ and there is nothing to prove.
	Suppose now that $\on{embdim}(S)>0$ and choose $f\in\m_{S}\setminus\m_{S}^2$. We claim that $f\in \m_R\setminus \m_R^2$.

	By Theorem \ref{Yuan-Harper theorem}, we can write $R$ as $L[[X_1,\ldots,X_n]]/\langle X_1^p,\ldots,X_n^p\rangle$ for some field $L$. Observe that Corollary \ref{cor: annihilator}, applied to $S$ and $f$, shows that $\on{ann}_S(f)=f^{p-1} S$, and given that $S\to R$ is flat, we also have $\on{ann}_R(f)=f^{p-1}R$. 
	Hence we have arrived to the situation where $R$ is a (zero-dimensional) complete intersection ring and the annihilator of $f$ in $R$ is a principal ideal. Under these conditions,
 	\cite[Proposition 3.3]{Kunz74} tells us that $R/(f)$ is also a complete intersection ring. In other words, if $F\in L[[X_1,\ldots,X_n]]$ is a representative of $f$, then the ideal $\langle X_1^p,\ldots,X_n^p,F\rangle\subset L[[X_1,\ldots,X_n]]$ is minimally generated by $n$ elements which can be chosen among the previous $n+1$ generators.
 	These can not be $X_1^p,\ldots,X_n^p$, since $f$ is not the zero element of $R$ (flatness of $R$ over $S$ implies that $S\to R$ is injective).
	Thus, we may assume without loss of generality that 
	$$\langle X_1^p,\ldots,X_n^p,F\rangle = \langle X_2^p,\ldots,X_n^p,F\rangle.$$
	Thus $X_1^p=GF+H_2X_2^p+\cdots +H_nX_n^p$ for some $G,H_2,\ldots,H_n\in L[[X_1,\ldots,X_n]]$. We use again that $\on{ann}_R(f)=f^{p-1}R$ and deduce that $G=UF^{p-1}+U_1 X_1^p+U_2 X_2^p+\ldots+U_n X_n^p$ for some $U,U_1,\ldots,U_n\in L[[X_1,\ldots,X_n]]$. Then it follows that 
	$$X_1^p\in \langle F^p, FX_1^p,X_2^p,\ldots,X_n^p\rangle.$$ 
	Viewing this in $\langle X_1,\ldots,X_n\rangle^p/\langle X_1,\ldots,X_n\rangle^{p+1}$, we see that necessarily $F\in \langle X_1,\ldots,X_n\rangle\setminus \langle X_1,\ldots,X_n\rangle^2$, hence $f\in\m_R\setminus \m_R^2$. This proves the claim.

	We finally set $\bar{S}:=S/fS$ and $\bar{R}:=R/fR$. The induced extension $\bar{S}\to \bar{R}$ is flat, it satisfies $\bar{R}^p\subset\bar{S}$ and $\bar{R}$ is differentiably simple by Proposition \ref{prop: quotient of differentiably simple rings}. Since $\on{embdim}(\bar{S})<\on{embdim}(S)$ we can use the inductive hypothesis to deduce that $\bar{R}/\m_{\bar{S}}\bar{R}$ is differentiably simple. Since $\bar{R}/\m_{\bar{S}}\bar{R}\cong R/\m_S R$, the ring $R/\m_S R$ is differentiably simple. This completes the proof.
\end{proof}

\begin{remark}
	Our proof of the fact that $R/\m_{S} R$ is differentiably simple \color{black} differs from the one given in \cite{Andre91}, which relies on a very strong result of Avramov (\cite[Theorem 1.1]{Avramov77}) about the homology of a general local flat extension. 
	The key point in our proof was to establish that $R/fR$ is a complete intersection ring. 
	We have used Kunz's paper \cite{Kunz74}, whose main result is based on a theorem of Gulliksen \cite[Proposition 1.4.9]{GullLev69}. Kunz also presents an alternative proof of his main result by using just the main technique in the proof of Gulliksen's theorem, which is Tate's method of adjoining variables to kill cycles \cite{Tate57}.  Given the simplicity of the ring $R$ one is tempted to claim that there should be a more elementary proof of Theorem \ref{th:flat homomorphism of diff simple rings}. However, we have not been able to find one.
\end{remark}

\begin{corollary}\label{cor: extending a minimal generating set of a subring}
	Let $(S,\m_S)\to (R,\m_R)$ be a flat extension of Noetherian differentiably simple rings such that $R^p\subset S$. Then any minimal generating set for the maximal ideal of $S$ can be extended to a minimal generating set for the maximal ideal of $R$.
\end{corollary}

\begin{proof}
	Let $\{x_1,\ldots,x_r\}$ be a minimal generating set for $\m_S$.
	Since $S\to R$ is flat, $\{x_1,\ldots,x_r\}$ is also a minimal generating set for the extended ideal $\m_S R$.
	Now, by assumption, $R$ is differentiably simple, and by Theorem \ref{th:flat homomorphism of diff simple rings}, $R/\m_S R$ is differentiably simple. 
	It follows from Proposition \ref{prop: quotient of differentiably simple rings} that $\m_S R$ is generated by part of a minimal generating set for $\m_R$.
	Since $\{x_1,\ldots,x_r\}$ is a minimal generating set for $\m_S R$ and since any two minimal generating sets for $\m_S R$ have the same length, we obtain that $\{x_1,\ldots,x_r\}$ is part of minimal generating set for $\m_R$.
\end{proof}

\section{Galois extensions of exponent one}\label{section: flat families}
Throughout this section, all rings are assumed to be of characteristic $p$.

In Section \ref{section: diff simple rings}, we have studied Noetherian differentiably simple rings. We now consider flat families of Noetherian differentiably simple rings. This notion was introduced by Yuan in \cite{Yuan70}, where they are presented under the name of purely inseparable Galois extension of exponent one.
In Section \ref{subsection: p-basis}, we introduce some preliminaries about $p$-basis. In Section \ref{subsection: Galois extensions}, we review the concept of purely inseparable Galois extension. Finally, in Section \ref{section: automorphisms}, we analyze certain variations of the automorphism group scheme of some specific Galois extensions.
The whole section is expository and serves as an introduction for the next and final section, where we introduce the Yuan scheme.
Most of the material in this section is included in \cite[\S 5 and \S 15]{Kunz86}.

\subsection{$p$-basis for ring extensions}\label{subsection: p-basis}

A ring extension $A\subset C$ is said to have {\em exponent one} if $x^p\in A$ for all $x\in C$, i.e., $C^p\subset A$. 
\begin{definition}\label{def: p-basis}
	Let $A\subset C$ be a finite ring extension of exponent one. 
	A collection $\{x_1,\ldots,x_n\}\subset C$ is called a \textit{$p$-basis} of $C$ over $A$ if $C$ is a free $A$-module with basis $\{x_1^{\alpha_1}\cdots x_n^{\alpha_n}: 0\leq \alpha_i<p\}$.  In other words,  if the $A$-algebra homomorphism $A[X_1,\ldots,X_n]\to C$ which maps $X_i$ to $x_i$ induces an isomorphism
\begin{align*}
A[X_1,\ldots,X_n]/\langle X_1^p-x_1^p,\ldots,X_n^p-x_n^p\rangle\cong C.
\end{align*}	
\end{definition}

For example, the existence of $p$-basis is guaranteed  when $A$ and $C$ are fields, and more generally when $A$ and $C$ are regular local rings (see \cite{KN82}).

The following lemma, that collects some basic properties of $p$-basis, is straightforward.
\begin{lemma}\label{lem: elementary properties of p-basis}
	Let $A\subset C$ be a ring extension of exponent one.
	\begin{enumerate}[(i)]
	\item If $\{x_1,\ldots,x_n\}$ is a $p$-basis of $C$ over $A$, then for any $a_1,\ldots,a_n\in A$, $\{x_1-a_1,\ldots,x_n-a_n\}$ is also a $p$-basis of $C$ over $A$.
	\item If $\{x_1,\ldots,x_n\}$ is a $p$-basis of $C$ over $A$ and $A'$ is an $A$-algebra, then $\{x_1\otimes 1,\ldots, x_n\otimes 1\}\subset C\otimes_A A'$ is  a $p$-basis for $C\otimes_A A'$ over $A'$. 
	\item Let $B$ be an intermediate ring, $A\subset B\subset C$. If $\{x_1,\ldots,x_r\}$ is a $p$-basis of $C$ over $B$ and $\{x_{r+1},\ldots,x_n\}$ is a $p$-basis of $B$ over $A$, then $\{x_1,\ldots,x_n\}$ is a $p$-basis of $C$ over $A$.
	\end{enumerate}
\end{lemma}

Let $A\subset C$ be a finite extension of exponent one which admits a $p$-basis. We say that the extension is {\em split} if $C\cong A[X_1,\ldots,X_n]/\langle X_1^p,\ldots,X_n^p\rangle$ as $A$-algebras. In other words, if there exists a $p$-basis $\{x_1,\ldots,x_n\}$ such that $x_i^p=0$ for $i=1,\ldots, n$. A $p$-basis with this property will be called {\em a splitting $p$-basis}.

\begin{proposition}\label{prop: extension with p-basis becomes trivial}
	Let $A\subset C$ be a ring extension of exponent one. Assume that $A\subset C$ has a $p$-basis $\{x_1,\ldots,x_n\}$.
	 Let  $A\subset A'$ be a ring extension such that, for each $i=1,\ldots,n$, there exists $y_i\in A'$ with $y_i^p=x_i^p$. Let $z_i:=x_i\otimes 1-1\otimes y_i$. Then $\{z_1,\ldots,z_n\}$ is a splitting $p$-basis of $A'\subset C\otimes_A A'$.
\end{proposition}
\begin{proof}
	By Lemma \ref{lem: elementary properties of p-basis}.(ii), $\{x_1\otimes 1,\ldots,x_n\otimes 1\}$ is a $p$-basis for $C\otimes_AA'$ over $A'$. Let $z_i=x_i\otimes 1-1\otimes y_i\in C\otimes_A A'$. By Lemma \ref{lem: elementary properties of p-basis}.(i), $\{z_1,\ldots,z_n\}$ is also a $p$-basis for $C\otimes_A A'$ over $A'$. Since $x_i^p=y_i^p\in A$, we have $z_i^p=0$, so $\{z_1,\ldots,z_n\}$ is a splitting $p$-basis. 
\end{proof}

\begin{remark}
	The hypothesis, for each $i=1,\dots,n$ there exists $y_i\in A'$ with $y_i^p=x_i^p$, is true if we take $A'$ to be $C$. It also applies if $A=k$ is a field and $A'$ is any other field including $k^{1/p}$. In particular, if $k$ is an algebraically closed field, then any finite extension $k\subset C$ admitting a $p$-basis is split.
\end{remark}

In Proposition \ref{prop: p-basis vs diff basis finite case} below, we review the fact that $\{x_1,\ldots,x_n\}$ is a $p$-basis for $A\subset C$ if and only if $\{dx_1,\ldots,dx_n\}$ is a basis of the $C$-module $\Omega_{C/A}$. We shall need the following result.

\begin{lemma}\label{lem: trivial omega}
	Let $C'\subset C$ be a finite extension of rings such that $C^p\subset C'$ and $\Omega_{C/C'}=0$. Then $C=C'$.
\end{lemma}

\begin{proof}
	Since $C$ is a finite $C'$-module,  in order to show that $C=C'$ it is enough to show that for all $\p'\in\Spec(C')$, the inclusion $\kappa(\p')\hookrightarrow C\otimes_{C'}\kappa(\p')$ is an equality, by Nakayama's Lemma. 
	Hence to prove the proposition we may assume that $C'$ is a field $k$. 
	Thus the assumption is that $C$ is a finite $k$-algebra such that $C^p\subset k$ (in particular $C$ is local) and $\Omega_{C/k}=0$, and we have to prove that $C=k$.
	Let $\m$ and $K$ be the maximal ideal and the residue field of $C$. 
	Note that $\Omega_{C/k}=0$ implies $\Omega_{K/k}=0$, and since $K^p\subset k$ we have $K=k$. 
	Thus, $C=k\oplus \m$. 
	Now let $D$ be the  composition of the projection $C\to \m$ followed by the quotient map $\m\to\m/\m^2$. Observe that $D$ is a surjective $k$-derivation from $C$ to $\m/\m^2$. Since $\Omega_{C/k}=0$ it must be $D=0$; thus $\m=\m^2$. Finally, by Nakayama's Lemma, $\m=0$ and the proof is complete.
\end{proof}

\begin{proposition}\label{prop: p-basis vs diff basis finite case}
	Let $A\subset C$ be a finite extension of exponent one and let $x_1,\ldots,x_n\in C$. Then $\{x_1,\ldots,x_n\}$ is a $p$-basis of $C$ over $A$ if and only if $\Omega_{C/A}$ is a free $C$-module with basis $\{dx_1,\ldots,dx_n\}$. 
\end{proposition}

\begin{proof}
	If $\{x_1,\ldots,x_n\}$ is a $p$-basis of $C$ over $A$, then $\{dx_1,\ldots,dx_n\}$ is a basis of  $\Omega_{C/A}$ as $C$-module by \cite[Proposition 5.6]{Kunz86}. 
	Conversely, assume that $\{dx_1,\ldots, dx_n\}$ is a basis of $\Omega_{C/A}$ as $C$-module and let $C':=A[x_1,\ldots,x_n]$. By \cite[Proposition 5.6]{Kunz86},
	$\{x_1,\ldots,x_n\}$ is a $p$-basis of $C'$  over $A$. 
	By the first implication, $dx_1,\ldots,dx_n$ generate $\Omega_{C'/A}$ and so the natural map $\Omega_{C'/A}\otimes_{C'}C\to \Omega_{C/A}$ is surjective (here, by abuse of notation, we denote by $d$ both the universal derivations of $C'$ and $C$ over $A$). It follows from the natural exact sequence $\Omega_{C'/A}\otimes_{C'}C\to\Omega_{C/A}\to \Omega_{C/C'}\to 0$ that $\Omega_{C/C'}=0$. Finally, $C=C'$ by Lemma \ref{lem: trivial omega}. This completes the proof of the converse.
\end{proof}

The next proposition and its corollary relate the notions of differentiably simple rings and $p$-basis.

\begin{proposition}\label{prop:p-basis over a field}
	Let $k\subset C$ be a finite extension of exponent one such that $k$ is a field. Then $C$ has $p$-basis over $k$ if and only if $C$ is differentiably simple.
\end{proposition}

\begin{proof}
	If $C$ has $p$-basis over $k$, then $C\otimes_k k^{1/p}\cong k^{1/p}[X_1,\ldots,X_n]/\langle X_1^p,\ldots,X_n^p\rangle$ for some $n$, by Proposition \ref{prop: extension with p-basis becomes trivial}.
	This ring is differentiably simple as shown in Example \ref{ex: example of diff simple}.
	Now $k^{1/p}$ is a free $k$-module, so $C\otimes_{k} k^{1/p}$ is a free $C$-module. Thus, $C$ is differentiably simple by Proposition \ref{subring of a differentiably simple ring}.

	Assume now that $C$ is differentiably simple. It is Noetherian since it is finite over the field $k$. Since $C^p\subset k$, the field $k$ is included in a coefficient field $L\subset C$ by Lemma \ref{Yuan's lemma}.
	Now, by Theorem \ref{Yuan-Harper theorem}, there is an isomorphism of $L$-algebras $C\cong L[X_1,\ldots,X_n]/\langle X_1^p,\ldots,X_n^p\rangle$. Observe that $C$ has a $p$-basis over $L$ (namely, the classes of $X_1,\ldots,X_n$) and $L$ has also a $p$-basis over $k$ since $k\subset L$ is a field extension. Thus, $C$ has a $p$-basis over $k$ by Lemma \ref{lem: elementary properties of p-basis}.(iii).
\end{proof}

\begin{corollary}\label{cor: p-basis and differentiably simple fiber}
	Let $A\subset C$ be a finite flat extension of exponent one between local rings. Then the following conditions are equivalent:
	\begin{enumerate}[(i)]
	\item $A\subset C$ has $p$-basis.
	\item $A/\m_A\subset C/\m_A C$ has $p$-basis.
	\item  $C/\m_A C$ is differentiably simple.
	\end{enumerate}
\end{corollary}

\begin{proof}
	The implication (i)$\Rightarrow$(ii) follows from Lemma \ref{lem: elementary properties of p-basis}.(ii). The equivalence (ii)$\Leftrightarrow$(iii) is Proposition \ref{prop:p-basis over a field}. It remains to show (ii)$\Rightarrow$(i).
	Since $C$ is flat over the local ring $A$, it is a free $A$-module of finite rank. Choose $x_1,\ldots,x_n\in C$, whose classes modulo $\m_A C$ form a $p$-basis for $A/\m_A\subset C/\m_A C$. Then, by definition, the classes in $C/\m_A C$ of the monomials in $\{x_1^{\alpha_1}\cdots x_n^{\alpha_n}: 0\leq\alpha_i<p\}$ form a basis for $C/\m_AC$ as vector space over $A/\m_A$. Thus, by Nakayama's Lemma, $\{x_1^{\alpha_1}\cdots x_n^{\alpha_n}: 0\leq\alpha_i<p\}$ is a basis of $C$ as free $A$-module, so $\{x_1,\ldots,x_n\}$ is a $p$-basis of $C$ over $A$.
\end{proof}

\subsection{Galois extensions of exponent one}\label{subsection: Galois extensions}

\begin{definition}[{\cite[Definition 6]{Yuan70}}]
	A ring extension $A\subset C$ is called {\em a purely inseparable Galois extension of exponent one}, or simply a Galois extension of exponent one, if
	\begin{enumerate}
		\item $A\subset C$ has exponent one, i.e. $C^p\subset A$,
		\item $C$ is finitely generated projective as $A$-module, and
		\item for each prime ideal $\p\subset A$, the ring $C_\p$ has $p$-basis over $A_{\p}$.
	\end{enumerate}
	By Corollary \ref{cor: p-basis and differentiably simple fiber}, Condition 3 can actually be replaced by the condition
	\begin{enumerate}
	\item[3'] for each prime ideal $\p\subset A$, the ring $C_\p/\p C_\p$ is differentiably simple.
	\end{enumerate}
\end{definition}

Observe that if $A\subset C$ is a Galois extension, then it is in particular faithfully flat. If $A'$ is any $A$-algebra, then the base change $A'\subset C\otimes_A A'$ is also a Galois extension since $p$-basis and the property of being a finitely generated projective module behave well under base extensions.

\begin{lemma}\label{lem: being Galois implies p-basis locally}
	Let $A\subset C$ be a finite ring extension of exponent one. Then $C$ is Galois over $A$ if and only if there exist $f_1,\ldots,f_s\in A$ generating the unit ideal $A$ such that each $C_{f_i}$ has a finite $p$-basis over $A_{f_i}$.
\end{lemma}

\begin{proof}
	The implication ($\Rightarrow$) is \cite[Lemma 7]{Yuan70}. The converse is trivial.
\end{proof}

\begin{proposition}\label{prop: Galois with omega}
	Let $A\subset C$ be a finite ring extension of exponent one. Then $C$ is Galois over $A$ if and only if $\Omega_{C/A}$ is a projective $C$-module.
\end{proposition}

\begin{proof}
	Assume that $C$ is Galois over $A$.
	By Lemma \ref{lem: being Galois implies p-basis locally}, there are $f_1,\ldots,f_s\in A$ generating the unit ideal such that $C_{f_i}$ has (finite) $p$-basis over $A_{f_i}$.
	By Proposition \ref{prop: p-basis vs diff basis finite case}, $\Omega_{C_{f_i}/A_{f_i}}=(\Omega_{C/A})_{f_i}$ is a free $C_{f_i}$-module of finite rank. Hence, $\Omega_{C/A}$ is a projective $C$-module.
	
	Conversely, assume that $\Omega_{C/A}$ is a projective $C$-module. It is finitely generated over $C$ since $A\subset C$ is finite. We now fix $\p\in \Spec(A)$. Note that $\Omega_{C_\p/A_\p}$ is a finite free $C_\p$-module and it has a basis of the form $\{dx_1,\ldots,dx_n\}$ for some $x_1,\ldots,x_n\in C$. Here $d\colon C\to\Omega_{C/A}$ denotes the universal derivation and we are viewing $dx_i\in \Omega_{C/A}$ as element of $\Omega_{C_\p/A_\p}=\Omega_{C/A}\otimes_C C_\p$.
	Now $\Omega_{C/A}$ is a finitely generated projective $C$-module and $\Spec(C)\to\Spec(A)$ is a homeomorphism, so there is $f\in A$ such that $dx_1,\ldots,dx_n$ form a basis for $\Omega_{C_f/A_f}=(\Omega_{C/A})_f$ as free $C_f$-module.
	It follows from Proposition \ref{prop: p-basis vs diff basis finite case} that $\{x_1,\ldots,x_n\}$ is a $p$-basis of $C_f$ over $A_f$.
	Since $\p$ was arbitrary, we conclude that there are $f_1,\ldots,f_s\in A$ generating the unit ideal such that $C_{f_i}$ has a $p$-basis over $A_{f_i}$. Thus, $A\subset C$ is Galois by Lemma \ref{lem: being Galois implies p-basis locally}.
\end{proof}

\begin{corollary}
	Let $A\subset C$ be a finite ring extension of exponent one.
	Let $A'$ be an $A$-algebra and set $C':=C\otimes_A A'$.
	If $C'$ is Galois over $A'$ and $A\to A'$ is faithfully flat, then $C$ is Galois over $A$.
\end{corollary}

\begin{proof}
	Note that $\Omega_{C/A}\otimes_CC'=\Omega_{C'/A'}$ and $\Omega_{C'/A'}$ is finitely generated projective as $C'$-module by Proposition \ref{prop: Galois with omega}. Note also that $C\to C'$ is faithfully flat.
	By faithfully flat descent, $\Omega_{C/A}$ is finitely generated projective as $C$-module, so $A\subset C$ is Galois by Proposition \ref{prop: Galois with omega}.
\end{proof}

\begin{proposition}[{\cite[Th\'eor\`eme 71]{Andre91}}]\label{prop: Galois intermediate extension}
	Let $A\subset C$ be a Galois extension of exponent one and let $B$ be an $A$-subalgebra of $C$. Then $B\subset C$ is Galois if and only if $C$ is projective as $B$-module.
	In that case, $A\subset B$ is also Galois. 
\end{proposition}

\begin{proof} 
	Note that both extensions $A\subset B$ and $B\subset C$ have exponent one. Since $C$ is finitely generated as $A$-module, it is so as $B$-module.

	If $B\subset C$ is Galois, then $C$ is projective as $B$-module by definition.	Conversely, if $C$ is projective as $B$-module, then Lemma \ref{lem: finite projective extensions} below shows that $B$ is a direct summand of $C$ as $B$-module. In particular, $B$ is a direct summand of $C$ as $A$-module, so it is finitely generated projective as $A$-module. So far we have checked that $A\subset B$ and $B\subset C$ satisfy conditions 1 and 2 of the definition of Galois extension. 
	Therefore to show that they are Galois, it only remains to check condition 3'. 
	In other words, we have to prove that for any $\p\in\Spec(A)$ and any $\q\in\Spec(B)$ with $\p=A\cap\q$, the rings $B_\p/\p B_\p$ and $C_\q/\q C_\q$ are differentiably simple. Note that the latter is equal to $C_\p/\q C_\p$, since there are natural identifications $\Spec(A)\simeq\Spec(B)\simeq\Spec(C)$ and $C_{\p}=C_{\q}$. 	
	
	Since $A\subset C$ is Galois, the ring $C_{\mathfrak{p}}/\mathfrak{p} C_{\mathfrak{p}}$ is differentiably simple. Since $B\to C$ is in particular flat, the extension  $B_{\mathfrak{p}}/\mathfrak{p} B_{\mathfrak{p}}\to C_{\mathfrak{p}}/\mathfrak{p} C_{\mathfrak{p}}$ is a flat local homomorphism of Noetherian local rings. 
	Observe that $\q B_\p/\p B_\p$ is the maximal ideal of $B_\p/\p B_\p$.
	Thus, by Theorem \ref{th:flat homomorphism of diff simple rings}, $B_{\mathfrak{p}}/\mathfrak{p} B_{\mathfrak{p}}$ and $C_\p/\q C_\p$ are differentiably simple, as we wanted to show.
\end{proof}

\begin{definition}
	A Galois extension of exponent one $A\subset C$ is said to have \textit{differential rank} $n$ if $\Omega_{C/A}$ has constant rank $n$ as $C$-module. Equivalently, if, for any $\p\in\Spec(A)$, the length of any  $p$-basis for $C_\p$ over $A_\p$ is $n$ (Proposition \ref{prop: p-basis vs diff basis finite case}).
\end{definition}

The following lemma is a direct consequence of Lemma \ref{lem: elementary properties of p-basis}.
\begin{lemma}\label{lem: on differential rank}
	Assume that $A\subset C$ is a Galois extension of exponent one and differential rank $n$. 
	\begin{enumerate}[(i)]
		\item For any $A$-algebra $A'$, the Galois extension $A'\subset C\otimes_A A'$ has also differential rank $n$.
		\item If $B\subset C$ is an $A$-subalgebra such that $C$ is Galois over $B$ of differential rank $r$, then the Galois extension $A\subset B$ (see Proposition \ref{prop: Galois intermediate extension}) has differential rank $n-r$.
	\end{enumerate}
 \end{lemma}

\begin{proposition}\label{prop: trivilization}
	Let $A\subset C$ be a Galois extension of differential rank $n$. Let $B\subset C$ be an $A$-subalgebra such that $B\subset C$ is a Galois extension of differential rank $r$. Then there exists a finitely presented and faithfully flat $A$-algebra $A'$ such that $A'\subset C\otimes_A A'$ has a splitting $p$-basis $z_1,\ldots,z_n$ such that  $B\otimes_A A'=A'[z_{r+1},\ldots,z_n]$.
\end{proposition}

\begin{proof} 
	By assumption, $B\subset C$ is a Galois extension of differential rank $r$, and $A\subset B$ is a Galois extension of differential rank $n-r$ by Proposition \ref{prop: Galois intermediate extension} and Lemma \ref{lem: on differential rank}. Since the induced morphisms $\Spec(C)\to\Spec(B)\to\Spec(A)$ are homeomorphisms, we may apply Lemma \ref{lem: being Galois implies p-basis locally} to both $A\subset B$ and $B\subset C$ and deduce that there are $f_1,\ldots,f_s\in A$ generating the unit ideal such that, for each $i=1,\ldots,s$, $B_{f_i}$ has a $p$-basis $\{x_{r+1}^{(i)},\ldots,x_n^{(i)}\}$ over $A_{f_i}$ and $C_{f_i}$ has a $p$-basis $\{x_1^{(i)},\ldots,x_r^{(i)}\}$ over $B_{f_i}$. Then it follows from Lemma \ref{lem: elementary properties of p-basis}.(iii) that $\{x_1^{(i)},\ldots,x_n^{(i)}\}$ is a $p$-basis of $C_{f_i}$ over $A_{f_i}$.
	
	We now apply Proposition \ref{prop: extension with p-basis becomes trivial} to the extension $A_{f_i}\subset C_{f_i}$, the $p$-basis $\{x_1^{(i)},\ldots,x_n^{(i)}\}$ and $A'=C_{f_i}$. Thus, the set $\{z_1^{(i)},\ldots,z_n^{(i)}\}$, where
	$$z_j^{(i)}:=x_j^{(i)}\otimes 1-1\otimes x_j^{(i)}\in C_{f_i}\otimes_{A_{f_i}} C_{f_i}=C\otimes_A C_{f_i},$$
	is a splitting $p$-basis of $C\otimes_A C_{f_i}$ over $C_{f_i}$. Clearly $B\otimes_A C_{f_i}=C_{f_i}[z_{r+1}^{(i)},\ldots,z_n^{(i)}]$.

	We now set $A':=C_{f_1}\times \cdots \times C_{f_s}$ and define $z_j:=(z_j^{(1)},\ldots,z_j^{(s)})\in C\otimes_A A'$ for $j=1,\ldots,n$. Note that $\{z_1,\ldots,z_n\}$ is a splitting $p$-basis of $C\otimes_A A'$ over $A'$ and that $B\otimes_A A'=A'[z_{r+1},\ldots,z_n]$.
	Now the localization $A\to A_{f_i}$ and the extension $A_{f_i}\subset C_{f_i}$ (induced from $A\subset C$) are flat of finite presentation, and hence, so is the composition $A\to C_{f_i}$. Then it follows that the induced homomorphism $A\to A'$ is flat of finite presentation. And, since the induced morphism $\Spec(A')\to\Spec(A)$ is clearly surjective, we deduce that $A'$ is a faithfully flat $A$-algebra. This completes the proof.
\end{proof}

\subsection{Automorphism group schemes of split Galois extensions}\label{section: automorphisms}

Let $A\subset C$ be any finite ring extension such that $C$ is projective as $A$-module. The automorphism group scheme $\Aut_{C/A}$ is the group valued functor
\begin{align*}
	\Aut_{C/A}\colon \on{Alg}_A&\to\on{Groups}\\
	R&\mapsto \on{Aut}_{R}(C\otimes_A R)\, .
\end{align*}
Here $\on{Alg}_A$ and $\on{Groups}$ denote, respectively, the category of $A$-algebras and the category of groups, and $\on{Aut}_{R}(C\otimes_A R)$ is the group of $R$-automorphisms of the $R$-algebra $C\otimes_A R$. 
This group-valued functor is in fact an affine group scheme of finite type over $A$ (cf. \cite[Ch. II, \S 1, 2.6]{DemGab70}).
In other words, there is a finitely generated $A$-algebra $H_{C/A}$ (the Hopf algebra of $\Aut_{C/A}$) and a natural isomorphism
\begin{align*}
\Aut_{C/A}(R)\cong \on{Hom}_A(H_{C/A},R),
\end{align*}
for any $A$-algebra $R$. 

A finite ring extension $A\subset C$ is called {\em purely inseparable} if there exists a faithfully flat $A$-algebra $A'$ such that $C\otimes_A A'=A'[X_1,\ldots,X_n]/\langle X_1^{p^{e_1}},\ldots,X_n^{p^{e_n}}\rangle$ for some positive integers $e_1,\ldots,e_n$.
Some examples are the usual finite purely inseparable field extensions (see \cite{Ras71}) and the Galois extensions of exponent one, by Proposition \ref{prop: trivilization}.
Their automorphism group schemes have been considered in the literature with different perspectives. For instance, B\'egueri \cite{Beg69}, Shatz \cite{Sha69} and Chase \cite{Cha72} discussed the case of finite purely inseparable field extensions in connection with Galois theory.
More recently, F. S. de Salas and P. S. de Salas \cite{FPSan00} considered $\Aut_{C/A}$ for arbitrary finite purely inseparable ring extensions. In particular, they obtained a criterion for when $\Aut_{C/A}$ is integral. 

For our applications in the construction of the Yuan scheme in Section \ref{section: Yuan scheme} we only need to consider some specific subgroups of the automorphism group scheme of a split Galois extension of exponent one $A\subset C$, that is, with $C$ of the form
$$C=A[X_1,\ldots,X_n]/\langle X_1^p,\ldots,X_n^p\rangle.$$
We let $x_i$ denote the class of $X_i$ and fix the splitting $p$-basis $\{x_1,\ldots,x_n\}$ for our next discussion. Let
$\m=\langle x_1,\ldots,x_n\rangle\subset C$,
and, for $0\leq r<n$, let
$$ B=A[x_{r+1},\ldots,x_n]\subset C.$$
We now define two group-valued subfunctors $\Aut_{C/A,\{B,\m\}}\subset\Aut_{C/A,\{B\}}\subset \Aut_{C/A}$, that attach to each $A$-algebra $R$ the sets
\begin{align*}
	\Aut_{C/A,\{B\}}(R)&:=\{\sigma\in\on{Aut}_R(C\otimes_A R)\colon \sigma(B\otimes_A R)=B\otimes_A R \},\\
	\Aut_{C/A,\{B,\m\}}(R)&:=\{\sigma\in\on{Aut}_R(C\otimes_A R)\colon \sigma(B\otimes_A R)=B\otimes_A R ,\ \sigma(\m\otimes_A R)=\m\otimes_A R\}.
\end{align*}

For the next proposition we denote by $\GL_m$ the general linear group viewed as scheme over $A$. That is, for an $A$-algebra $R$, $\GL_m(R)$ is the set of invertible $(m\times m)$-matrix with entries in $R$.
Conventionally, $\GL_0$ is the trivial $A$-scheme $\Spec(A)$.

\begin{proposition}
	$\Aut_{C/A,\{B\}}$ and $\Aut_{C/A,\{B,\m\}}$ are representable by closed subgroups of $\Aut_{C/A}$. Moreover,
	\begin{enumerate}
		\item $\Aut_{C/A,\{B,\m\}}$ is isomorphic as scheme over $A$ to $\GL_{r}\times \GL_{n-r}\times\mathbb{A}_A^{t}$, where $$t+r^2+(n-r)^2=(p^n-1)r+(p^{n-r}-1)(n-r).$$
		\item  $\Aut_{C/A,\{B\}}\cong \Aut_{C/A,\{B,\m\}}\times\Spec(C)$ as schemes over $A$.
	\end{enumerate}
\end{proposition}

\begin{proof}
		We use the special notation $\e_1,\ldots,\e_n$ for the canonical vectors of $\mb{N}_0^n$ and write $\0:=(0,\ldots,0)$. We set
		\begin{align*}
			\mc{A}:=&\{\alpha=(\alpha_1,\ldots,\alpha_n)\in\mb{N}_0^n: 0\leq\alpha_i<p\},
			&\mc{A}_+:=\mc{A}\backslash\{\0\},\\
			\mc{B}:=&\{\beta\in\mc{A}: \beta_i=0\quad\mbox{for } i\leq r\},
			&\mc{B}_+:=\mc{B}\cap\mc{A}_+.
		\end{align*}

		Fix an $A$-algebra $R$. 
		An endomorphism of $R$-algebras $\tau\colon C\otimes_A R\to C\otimes_A R$ satisfying $\tau(B\otimes_A R)\subset B\otimes_A R$
		is completely determined by its values $\tau(x_1),\ldots,\tau(x_n)$, which can be any elements $f_1,\ldots,f_n\in C\otimes_A R$ such that $f_i^p=0$, for $i=1,\ldots,n$, and $f_{r+1},\ldots,f_n\in B\otimes_A R$. 
		In other words, $f_1,\ldots,f_n$ can be any elements of the form
			\begin{align*}
			f_i&:=c_{i,\0}+ \sum_{\alpha\in\mc{A}_+} c_{i,\alpha}\x^\alpha,\quad \text{for } i=1,\ldots,r,\\
			f_i&:= c_{i,\0}+ \sum_{\beta\in\mc{B}_+} c_{i,\beta}\x^\beta,\quad \text{for } i=r+1,\ldots,n,
		\end{align*} 
	with $c_{i,\0}^p=0$ for $i=1,\ldots,n$. 
	
	We now determine when an $R$-endomorphism $\tau$ defined by $f_1,\ldots,f_n$ as above is bijective. Since $C\otimes_A R$ is a finitely generated $R$-module, this is the same as requiring that $\tau$ is surjective, or say, that $f_1,\ldots,f_n$ generate $C\otimes_A R$ as an $R$-algebra.
	This is in turn equivalent to saying that $f_i^+:=f_i-c_{i,\0}$, $i=1,\ldots,n$, generate $C\otimes_A R$ as $R$-algebra. 
	Since $\m\otimes_A R$ is nilpotent, this is the same as requiring that the classes of $f_1^+,\ldots,f_n^+\in\m\otimes_A R$ modulo $\m^2\otimes_A R$ generate the $R$-module $(\m/\m^2)\otimes_A R$, which is a free $R$-module of rank $n$. 
	Thus,  $\tau$ is bijective if and only if $(c_{i,\mathbf{e}_j})_{1\leq i,j\leq n}\in \GL_n(R)$. 
		Since $c_{i,\mathbf{e}_j}=0$ if $j\leq r$ and $i>r$, we have 
		$(c_{i,\mathbf{e}_j})_{1\leq i,j\leq n}\in \GL_n(R)$ if and only if $(c_{i,\mathbf{e}_j})_{1\leq i,j\leq r}\in \GL_r(R)$ and $(c_{i,\mathbf{e}_j})_{r+1\leq i,j\leq n}\in \GL_{n-r}(R)$.
		
		Note finally that we can identify $(c_{1,\0},\ldots,c_{n,\0})$ with a point of $\on{Hom}_A(C,R)$, the $A$-algebra homomorphism which maps $x_i$ to $c_{i,\0}$, and that we have $\tau(\m\otimes_A R)=\m\otimes_A R$ precisely when $c_{1,\0}=\cdots=c_{n,\0}=0$.
	
	All this shows that there are natural bijections
		\begin{align*}
			\Aut_{C/A,\{B\}}(R)&\cong \on{Hom}_A(C,R)\times \GL_r(R)\times \GL_{n-r}(R)\times \prod_{i=1}^r R^{\mc{A}'}\times\prod_{i=r+1}^n R^{\mc{B}'},\\
			\Aut_{C/A,\{B,\m\}}(R)&\cong \on{Hom}_A(C/\m,R)\times \GL_r(R)\times \GL_{n-r}(R)\times \prod_{i=1}^r R^{\mc{A}'}\times\prod_{i=r+1}^n R^{\mc{B}'},
		\end{align*}
		where $\mc{A}':=\mc{A}_+\backslash\{\e_{1},\ldots,\e_r\}$ and $\mc{B}':=\mc{B}_+\backslash\{\e_{r+1},\ldots,\e_n\}$. 
		In other words, $\Aut_{C/A,\{B\}}$ is representable by the affine $A$-scheme
		$\Spec(C)\times\GL_r\times\GL_{n-r}\times \mb{A}_A^t$, where $t=r|\mc{A}'|+(n-r)|\mc{B}'|$, and $\Aut_{C/A,\{B,\m\}}$ is represented by the closed subscheme $\GL_r\times\GL_{n-r}\times \mb{A}_A^t$. The proof follows as
		$|\mc{A}'|=p^n-r-1$ and $|\mc{B}'|=p^{n-r}-1-(n-r)$ and therefore
		\begin{align*}
			r^2+(n-r)^2+t&=r^2+(n-r)^2+r(p^n-r-1) +(n-r)(p^{n-r}-1-(n-r))\\
			&=(p^n-1)r+(p^{n-r}-1)(n-r)\,.
		\end{align*}
\end{proof}

\begin{remark}
	Since $\on{Hom}_A(C,R)=\{(c_1,\ldots,c_n)\in R^n\colon c_i^p=0\}$ is an additive subgroup of $R^n$, $\Spec(C)$ is actually an affine group scheme. In fact, $\Spec(C)$ can be identified with a subgroup of $\Aut_{C/A,\{B\}}$. Namely, an $R$-point $(c_1,\ldots,c_n)\in \Spec(C)(R)=\on{Hom}_A(C,R)$ can be identified with the $R$-automorphism $\tau\in \Aut_{C/A,\{B\}}(R)$ such that $\tau(x_i)=c_i+x_i$. 
\end{remark}

The previous proposition tells us that $\Aut_{C/A,\{B\}}$ is isomorphic to the $A$-scheme
$\Spec(C)\times\GL_r\times\GL_{n-r}\times \mb{A}_A^t$. So, for example, if $A$ is a reduced ring, then all factors are reduced except for the first, that corresponds with the ring $C$, which is non-reduced if $n>0$. In fact, in this case, the reduced scheme of $\Aut_{C/A,\{B\}}$ is $\Aut_{C/A,\{B,\m\}}$.

\begin{corollary}\label{cor: about Aut}
	If $A=k$ is a field, then $\Aut_{C/k,\{B,\m\}}$ is smooth and geometrically irreducible of dimension $(p^n-1)r+(p^{n-r}-1)(n-r)$. In addition, it is the reduced scheme of $\Aut_{C,\{B\}}$.
\end{corollary}

\section{The Yuan scheme}\label{section: Yuan scheme}

Let $A\subset C$ be a Galois extension of exponent one and let $\g:=\on{Der}_A(C)$ be the $C$-module of $A$-derivations. Then $\Omega_{C/A}$ is a finitely generated projective $C$-module (Proposition \ref{prop: Galois with omega}) and hence, so is $\g=(\Omega_{C/A})^*$. Recall that given $D, D'\in \g$, the Lie bracket $[D,D']:=D\circ D'-D'\circ D$ and  the $p$-times composition $D^p:=D\circ\cdots\circ D$ are again $A$-derivations. These operations make $\g$ into a restricted Lie algebra over $A$. 

To any $A$-subalgebra $B\subset C$ one associates a $C$-submodule $\g_{C/B}:=\on{Der}_B(C)\subset \g$. This submodule is closed under the Lie bracket and the operation $D\mapsto D^p$. In other words, $\g_{C/B}$ is both a $C$-submodule and a restricted Lie $A$-subalgebra of $\g$. Conversely, any $C$-submodule $\h\subset \g$ that is also a restricted Lie $A$-subalgebra (and more generally, any subset of $\g$) defines an $A$-subalgebra of $C$, namely $\ker(\h):=\{x\in C: D(x)=0\quad\forall D\in \h\}$.

In \cite[Theorem 13]{Yuan70}, Yuan proved that the application $B\mapsto \g_{C/B}$ defines a one-to-one correspondence between the $A$-subalgebras $B\subset C$ such that $C$ is also Galois over $B$ and the restricted Lie $A$-subalgebras $\h\subset \g$ that are $C$-module direct summand of $\g$.
The inverse of this correspondence is given by $\h\mapsto\ker(\h)$.

In this section, we consider the family of all $A$-subalgebras $B\subset C$ such that $C$ is Galois over $B$ of a fixed rank, and our task is to parameterize this family with an $A$-scheme.

\begin{definition}
	Fix a Galois extension $A\subset C$ of differential rank $n$ and an integer $r$, $0\leq r\leq n$.  The \textit{Yuan functor} associated to $C$ and $r$ is the following covariant functor from the category of $A$-algebras to the category of sets 
	\begin{align*}
		\mc{Y}_{C/A}^r \colon \on{Alg}_A&\to\on{Sets}\\
		A'&\mapsto\left\{B': \begin{array}{l}
		A'\subset B'\subset C\otimes_A A', \\ \text{ $B'\subset C\otimes_A A'$ is Galois of differential rank $r$}
		\end{array}\right\}\, .
	\end{align*} 
\end{definition}

Recall that, given  $B'\in\Y_{C/A}^r(A')$, $A'\subset B'$ is Galois of differential rank $n-r$
(Proposition \ref{prop: Galois intermediate extension} and Lemma \ref{lem: on differential rank}).
 
\begin{remark}\label{rem: base change of Yuan}
	By the nature of the definition, if $A'$ is an $A$-algebra and $C':=C\otimes_A A'$, then $\Y_{C'/A'}^r=\Y_{C/A}^r\times_A\Spec(A')$.
\end{remark}

The goal of this section is to prove first that $\Y_{C/A}^r$ is representable by a quasi-projective $A$-scheme and then to study some relevant geometric properties.
Our proof of the fact that $\Y_{C/A}^r$ is representable will follow from a more general result, Theorem \ref{th: representability in general}, which establishes the representability of a large class of functors including the Yuan functors. We introduce some preliminaries for the proof of this general theorem.

\subsection{Preliminaries on projective modules}

We collect some lemmas on projective modules for the sake of convenience of the readers. In this section we do not assume that the rings have characteristic $p$. Recall that for an $A$-module $M$ the following are equivalent \cite[Proposition B.27]{GorWed10}:
\begin{itemize}
	\item $M$ is a finitely generated projective $A$-module.
	\item $M$ is a finitely presented $A$-module and for each $\p\in\Spec(A)$, the $A_\p$-module $M_\p$ is free.
\end{itemize}
A finitely generated projective $A$-module $M$ is said to have \textit{rank} $m$ if for each prime $\p\subset A$, the free $A_\p$-module $M_\p$ has rank $m$. If $M$ is projective of rank $m$ as $A$-module, then for any $A$-algebra $A'$,  $M\otimes_A A'$ is projective of rank $m$ as $A'$-module.

Our main tool to detect whether a finitely presented $A$-module $M$ is projective will be given by the  {\em Fitting ideals}
$$0\subset \on{Fitt}_0(M)\subset \on{Fitt}_1(M)\subset\cdots \subset A.$$
For the definition and a detailed introduction we refer to \cite[Appendix D]{Kunz86}. We write $\on{Fitt}_i(M)=\on{Fitt}_i^A(M)$ when we want to emphasize the base ring. The Fitting ideals behave well under localization and base change in the sense that if $A'$ is an $A$-algebra, then $\on{Fitt}^{A'}_i(M\otimes_A A')$ is the extended ideal of $\on{Fitt}^A_i(M)$ under $A\to A'$ (\cite[Appendix D.4]{Kunz86}).
The key property will be the following characterization: 

\begin{lemma}\label{lem: projective modules via Fitting ideals}
	Let $M$ be a finitely presented $A$-module.
	Then $M$ is projective of rank $m$ if and only if $\on{Fitt}^A_i(M)=0$, for $i<m$, and $\on{Fitt}_m(M)=A$.
\end{lemma}

\begin{proof}
	This is a consequence of \cite[Appendix D.16]{Kunz86} and the assumption of finite presentation.
\end{proof}

We now discuss properties that hold for finite ring extensions $A\subset C$ such that $C$ is a projective $A$-module.

\begin{lemma}\label{lem: finite projective extensions}
	Let $A\subset C$ be a finite extension of rings such that $C$ is a projective $A$-module. Then the quotient $C/A$ is projective as $A$-module and, thus, the sequence $0\to A \to C \to C/A\to 0$ is split exact. 
\end{lemma}

\begin{proof}
	Since $C$ is a finitely generated projective $A$-module, for each $\p\in\Spec(A)$, $C_\p$ is a free $A_\p$-module of finite rank. 
	By Nakayama's Lemma, a basis for $C_\p$ over $A_\p$ can be obtained by lifting any basis of $C_\p/\p C_\p$ over $A_\p/\p A_\p$. A basis for this vector space can be always chosen so as to contain the vector $1$. Hence $C_\p$ has a basis over $A_\p$ containing 1. It follows that $A_\p$ is a direct summand of $C_\p$ and therefore $(C/A)_\p=C_\p/A_\p$ is a free $A_\p$-module.
	 Finally, $C/A$ is finitely presented as $A$-module since $C$ and $A$ are so. Thus $C/A$ is projective as $A$-module.
\end{proof}

\begin{lemma}\label{lem: intermediate subring}
	Let $A\subset C$ be a finite extension of rings such that $C$ is a projective $A$-module of rank $m$. 
	Fix an $A$-subalgebra $B\subset C$ such that $C$ is a projective $B$-module of rank $q$. Then
	\begin{enumerate}
		\item $q|m$ and $B$ is a finitely generated projective $A$-module of rank $\frac{m}{q}$.
		\item The quotient module $C/B$ is a projective $A$-module of rank $m-\frac{m}{q}$.
	\end{enumerate}
\end{lemma}

\begin{proof}
	By Lemma \ref{lem: finite projective extensions}, $B$ is a direct summand of $C$ as a $B$-module and hence also as an $A$-module. Since $C$ is finitely generated projective as $A$-module, both $B$ and $C/B$ are also finitely generated projective $A$-modules. We will now show that $B$ has rank $s$ with $qs=m$. This clearly implies the assertions about the ranks.

	After tensoring with $\kappa(\p)$ for some $\p\in\Spec(A)$, we may assume that $A$ is a field.	Let $s$ be the dimension of $B$ over $A$, so that $B\cong A^s$ as $A$-modules. Note that $B$ is an Artinian ring and $C$ is a finitely generated projective $B$-module of constant rank $q$. Thus, $C\cong B^q$ as $B$-modules. 
	Finally, $C\cong A^m$ as $A$-modules and therefore $A^m\cong C\cong B^q\cong(A^s)^q$, which implies $m=qs$. This completes the proof of the lemma. 
\end{proof}

As we shall see later, some arguments concerning finite extensions, $A\subset C$, can be simplified when $A$ is Noetherian. The following result will be helpful.

\begin{lemma}\label{lem: reduction to a neotherian situation}
	Let $A\subset C$ be a finite extension such that $C$ is projective of rank $m$ as $A$-module.
	Then there is a Noetherian subring $A_0\subset A$ and a ring extension $A_0\subset C_0$ such that $C_0$ is finitely generated projective of rank $m$ as $A_0$-module and $C=C_0\otimes_{A_0} A$. 
\end{lemma}

\begin{proof}
	Since $C$ is a finitely presented as $A$-module, it is a finite $A$-algebra of finite presentation \cite[Propostion B.12]{GorWed10}.
	By expressing $A$ as the inductive limit of its Noetherian $\mb{Z}$-subalgebras and applying \cite[Lemma B.9 and Proposition B.52]{GorWed10} we obtain a Noetherian subring $A_0'$ of $A$ and a finite $A_0'$-algebra $C_0'$ such that $C_0'\otimes_{A_0'}A=C$. Note that $C_0'$ is an $A_0'$-module of finite presentation because $A_0'$ is Noetherian.

    Now $C$ is projective of rank $m$ as $A$-module, so that  $\on{Fitt}^A_i(C)=0$ for $i<m$ and $\on{Fitt}^A_m(C)=A$ by Lemma \ref{lem: projective modules via Fitting ideals}. Since $\on{Fitt}^{A_0'}_m(C_0')A=\on{Fitt}^A_m(C)=A$, there are $f_1,\ldots,f_s\in \on{Fitt}^{A_0'}_m(C_0')$ and $h_1,\ldots,h_s\in A$ such that $f_1h_1+\cdots+f_sh_s=1$. We define $A_0$ as the $A_0'$-subalgebra of $A$ generated by $h_1,\ldots,h_s$, and set $C_0:=C_0'\otimes_{A_0'} A_0$.  It follows that $C_0$ is an $A_0$-module of finite presentation.
	
	We claim that $A_0\to C_0$ satisfies the required conditions.
	It only remains to show that $C_0$ is a projective $A_0$-module of rank $m$ (which implies in particular that $A_0\to C_0$ is injective).
 	To establish this we shall use Lemma \ref{lem: projective modules via Fitting ideals}.
	On the one hand, $\on{Fitt}^{A_0}_m(C_0)=\on{Fitt}^{A_0'}_m(C_0')A_0$, and this is equal to $A_0$ by construction of $A_0$.
	On the other hand, if $i<m$ then $0=\on{Fitt}^A_i(C)=\on{Fitt}^{A_0}_i(C_0)A$, hence $\on{Fitt}^{A_0}_i(C_0)=0$ as $A_0$ is a subring of $A$. Thus, $C_0$ is a projective $A_0$-module of rank $m$ by Lemma \ref{lem: projective modules via Fitting ideals}.
\end{proof}

\subsection{Grassmannian of subalgebras}

Let $A\subset C$ be a finite extension of rings (not necessarily of characteristic $p$) such that $C$ is a projective $A$-module of rank $m$. 
Fix a positive divisor $q|m$. Then the association 
\begin{align*}
	\mc{G}_{C/A}^q\colon \on{Alg}_A&\to \on{Sets}\\
	A'&\mapsto \left\{B': \begin{array}{l} \text{$B'$ is $A'$-subalgebra of $C':=C\otimes_A A'$},\\ \text{$C'$ is a projective $B'$-module of rank $q$} \end{array} \right\}
\end{align*}
defines a covariant functor.

 In fact, Lemma \ref{lem: intermediate subring} 
shows that $\mc{G}_{C/A}^q$ is a subfunctor of the Grassmannian functor
\begin{align*}
	\on{Grass}_{C/A}^{m-\frac{m}{q}}\colon \on{Alg}_A&\to\on{Sets}\\
	A'&\mapsto \left\{ B': \begin{array}{l}
	\text{$B'$ is an $A'$-submodule  
	of $C':=C\otimes_A A'$},\\
	\text{$C'/B'$ is a projective $A'$-module of local rank $m-\frac{m}{q}$}
	 \end{array} \right\}
\end{align*} 
It is well-known that $\on{Grass}_{C/A}^{m-\frac{m}{q}}$ is representable by a projective $A$-scheme \cite[Proposition 8.14]{GorWed10}.

\begin{remark}\label{rem: base change of G}
	By the nature of the definition, if $A'$ is an $A$-algebra and $C':=C\otimes_A A'$, then the inclusion of functors $\mc{G}_{C'/A'}^q\subset \on{Grass}_{C'/A'}^{m-\frac{m}{q}}$ is the base change of the inclusion $\mc{G}_{C/A}^q\subset \on{Grass}_{C/A}^{m-\frac{m}{q}}$. 
\end{remark}

\begin{theorem}\label{th: representability in general}
	$\mc{G}_{C/A}^q$ is representable by a quasi-projective scheme over $A$.
	More precisely, the inclusion of functors $\mc{G}_{C/A}^q\subset \on{Grass}_{C/A}^{m-\frac{m}{q}}$ can be factorized as the composition of a quasi-compact open immersion followed by a closed immersion.
\end{theorem}

\begin{proof}
	By Lemma \ref{lem: reduction to a neotherian situation} and Remark \ref{rem: base change of G}, it is enough to prove the theorem under the assumption that $A$ is Noetherian.
	We will express the inclusion $\mc{G}_{C/A}^q\subset \on{Grass}_{C/A}^{m-\frac{m}{q}}$ as a composition  $$\mc{G}_{C/A}^q\to \mc{G}_2\to\mc{G}_1\to \on{Grass}_{C/A}^{m-\frac{m}{q}}\,,$$ where $\mc{G}_1\to \on{Grass}_{C/A}^{m-\frac{m}{q}}$ and $\mc{G}_2\to\mc{G}_1$ will be closed immersions and $\mc{G}_{C/A}^q\to \mc{G}_2$ will be an open immersion, which is necessarily quasi-compact since $A$ is Noetherian.

	We first consider the functor
	\begin{align*}
		\mc{G}_1\colon \on{Alg}_A&\to\on{Sets}\\
		A'&\mapsto  \left\{B'\in \on{Grass}_{C/A}^{m-\frac{m}{q}}(A'): \text{$B'$ is an $A'$-subalgebra of $C\otimes_A A'$} \right\}.
	\end{align*}
	We prove that the inclusion $\mc{G}_1\subset \on{Grass}_{C/A}^{m-\frac{m}{q}}$ is a closed immersion. Let $\pi\colon \on{Grass}_{C/A}^{m-\frac{m}{q}}\to\Spec(A)$ be the structural morphism and let $\ms{B}\subset\pi^*C$ be the universal bundle.
	So for any $A$-algebra $A'$, we have a natural bijection
	\begin{align}\label{eq: natural bijection by the universal bundle}
		\on{Hom}_A(\Spec(A'), \on{Grass}_{C/A}^{m-\frac{m}{q}})&\cong \on{Grass}_{C/A}^{m-\frac{m}{q}}(A')\\
		\nonumber h&\mapsto h^*\ms{B}\, .
	\end{align}
	We consider the following compositions of natural morphisms of sheaves on $\on{Grass}_{C/A}^{m-\frac{m}{q}}$:
	\begin{align*}
		v&\colon \ms{B}\otimes\ms{B}\to \pi^*C\otimes\pi^*C{\to} \pi^*C\to \pi^*C/\ms{B}\,, \\
		w&\colon \pi^*A\to \pi^*C\to \pi^*C/\ms{B}\,,
	\end{align*}
	where the middle morphism in the definition of $v$ is induced by the multiplication of the $A$-algebra $C$. 
	Given an $A$-algebra $A'$ and $h\in\on{Hom}_A(\Spec(A'),\on{Grass}_{C/A}^{m-\frac{m}{q}})$, the condition $h^*v=0$ asserts that the $A'$-submodule $h^*\ms{B}\subset C\otimes_A A'$ is closed under the multiplication of $C'$, and the condition $h^*w=0$ asserts that $A'\subset h^*\ms{B}$. 
	Thus, having $h^*v=0$ and $h^*w=0$ is equivalent to $h^*\ms{B}\in\mc{G}_1(A')$.
	Since $\pi^*C/\ms{B}$ is finite locally free, $\mc{G}_1$ is representable and the inclusion $\mc{G}_1\subset \on{Grass}_{C}^{m-\frac{m}{q}}$ is a closed immersion, by \cite[Proposition 8.4]{GorWed10}.

	We now consider the subfunctor
	\begin{align*}
		\mc{G}_2\colon \on{Alg}_A&\to\on{Sets}\\
		A'&\mapsto \{B'\in\mc{G}_1(A'): \text{$\on{Fitt}^{B'}_i(C\otimes_A A')=0$ for $i< q$}\}.
	\end{align*}
	We prove that the inclusion $\mc{G}_2\subset\mc{G}_1$ is a closed immersion.
	Let $\pi\colon\mc{G}_1\to\Spec(A)$ be the structural morphism and let $\ms{B}_1\subset\pi^*C$ be the universal subalgebra. So for any $A$-algebra $A'$, we have a natural bijection
	\begin{align}\label{eq: natural bijection via the universal subring}
		\on{Hom}_A(\Spec(A'), \mc{G}_1)&\cong \mc{G}_1(A')\\
		\nonumber h&\mapsto h^*\ms{B}_1\, .
	\end{align}
	For each $i=0,1,2,\dots$, we let $v_i\colon\on{Fitt}^{\ms{B}_1}_i(\pi^* C)\to \ms{B}_1$ denote the inclusion of the $i$-th Fitting ideal of the $\ms{B}_1$-module $\pi^*C$. 
	Given an $A$-algebra $A'$ and an $A$-morphism $h\colon\Spec(A')\to\mc{G}_1$, by the good behavior of the Fitting ideals under base change, the Fitting ideal $\on{Fitt}^{h^*\ms{B}_1}_i(C\otimes_A A')$ is the image of the pullback $h^*v_i$. Hence, under the bijection (\ref{eq: natural bijection via the universal subring}), 
	$\mc{G}_2(A')\subset\mc{G}_1(A')$ corresponds to those $A$-morphisms $h$ such that $h^*v_i=0$ for $i=0,1,\ldots,q-1$. 
	Since $\ms{B}_1$ is finite locally free, $\mc{G}_2$ is representable and the inclusion $\mc{G}_2\subset\mc{G}_1$ is a closed immersion, by \cite[Proposition 8.4]{GorWed10}.

	We finally study $\mc{G}_{C/A}^q\to \mc{G}_2$. Observe that if $B'\in\on{Grass}_{C/A}^{n-\frac{n}{r}}(A')$, then $B'$ is a finitely generated projective $A'$-module since both $C':=C\otimes_A A'$ and $C'/B'$ are so.
	If in addition $B'\in\mc{G}_1(A')$, then $C':=C\otimes_A A'$ is a finitely presented $B'$-module. Indeed, $C'$ is clearly finitely generated as $B'$-module, so there is an exact sequence of $B'$-modules $0\to K\to {B'}^d\to C'\to 0$ for some  positive integer  $d$. This is also an exact sequence of $A'$-modules. 
	Now $C'$ is a projective $A'$-module, so it is of finite presentation, and since $B'$ is a finitely generated projective $A'$-module, we deduce that $K$ is a finitely generated $A'$-module. In particular, $K$ is a finitely generated $B'$-module and $C'$ is a $B'$-module of finite presentation.

	It follows from the above argument and Lemma \ref{lem: projective modules via Fitting ideals} that
	 $\mc{G}_{C/A}^q$ is the subfunctor of $\mc{G}_2$ given by
	\begin{align*}
		\mc{G}_{C/A}^q\colon \on{Alg}_A&\to\on{Sets}\\
		A'&\mapsto\{B'\in\mc{G}_2(A'): \on{Fitt}^{B'}_q(C\otimes_A A')=B'\}\,.
	\end{align*}
	We prove that the inclusion $\mc{G}_{C/A}^q\subset\mc{G}_2$ is an open immersion.
	Let $\pi\colon\mc{G}_2\to\Spec(A)$ be the structural morphism and let $\ms{B}_2\subset\pi^*C$ be the universal bundle. So for any $A$-algebra $A'$, we have a natural bijection
	\begin{align}\label{eq: natural bijection via universal subring II}
		\on{Hom}_A(\Spec(A'), \mc{G}_2)&\cong \mc{G}_2(A')\\
\nonumber		h&\mapsto h^*\ms{B}_2\, .
	\end{align} 
	Consider now the map $v_q\colon \on{Fitt}_q^{\ms{B}_2}(\pi^* C)\to \ms{B}_2$. By the same token as in the proof that $\mc{G}_2\subset \mc{G}_1$ is a closed immersion, one can show that, under the natural bijection (\ref{eq: natural bijection via universal subring II}), 
	$\mc{G}_{C/A}^q(A')\subset \mc{G}_{2}(A')$ is in correspondence with those $h$ such that $h^*v_q$ is surjective.
	By \cite[Proposition 8.4]{GorWed10}, the inclusion $\mc{G}_{C/A}^q\hookrightarrow\mc{G}_2$ is an open immersion. This completes the proof of the theorem. 
\end{proof}

\begin{theorem}\label{th: existence of the Yuan scheme}
	Let $A\subset C$ be a Galois extension of differential rank $n$. Fix a non-negative integer $r\leq n$. Then the Yuan scheme $\Y_{C/A}^r$ is representable by a quasi-projective $A$-scheme.
\end{theorem}

\begin{proof}
	Note that a Galois extension of exponent one $S\subset R$ has differential rank $s$ if and only if $R$ has rank $p^s$ as $S$-module. In particular, $C$ has rank $m=p^n$ as $A$-module. This observation and Proposition \ref{prop: Galois intermediate extension} show that $\Y_{C/A}^r=\mc{G}_{C/A}^{q}$ with $q=p^r$. 
	Thus, the representability of $\Y_{C/A}^r$ follows from Theorem \ref{th: representability in general}.
\end{proof}

The following example shows that the Yuan scheme $\Y_{C/A}^r$ might not be projective over $A$.
\begin{example}
	Let $k=\overline{\mb{F}_2}$ and $C=k[X,Y]/\langle X^2, Y^2\rangle=k\oplus kx\oplus ky\oplus kxy$, where $x$ and $y$ denote the class of $X$ and $Y$, respectively.
	An element of $\Y_{C/k}^1(k)$ is a subextension $k\subset B \subset C$  such that $C$ is flat over $B$.
	As $C$ is differentiably simple, $B$ is also differentiably simple (by Proposition \ref{subring of a differentiably simple ring}).
	Moreover, by Corollary \ref{cor: extending a minimal generating set of a subring}, $B$ has the form $k\oplus k(ax+by+cxy)$ for some $a,b,c\in k$ with $a\neq 0$ or $b\neq 0$. Hence  $\Y_{C/k}^1(k)$ is isomorphic as variety to $\{[a:b:c]\in\mb{P}^2(k): a\neq 0\ \mbox{or}\ b\neq 0\}\subset\mb{P}^2(k)$. This is clearly not a projective variety.
\end{example}

\subsection{Geometric properties of the Yuan scheme}

Let $A\subset C$ be a Galois extension of differential rank $n$. Fix a non-negative integer $r\leq n$. We shall analyze the geometric properties of the $A$-scheme $\Y_{C/A}^r$. 
A natural way to start is by analyzing the fibers of the structural morphism $\Y_{C/A}^r\to\Spec(A)$.
By Remark \ref{rem: base change of Yuan}, this is the same as analyzing $\Y_{C/A}^r$ when $A=k$ is a field. 
For this we consider the natural action of the algebraic group $\Aut_{C/k}$ on $\Y_{C/k}^r$ given for any $k$-algebra $R$ by
\begin{align*}
	\Aut_{C/k}(R)\times \Y_{C/k}^r(R)&\to \Y_{C/k}^r(R)\\
	(\tau,S)&\mapsto \tau(S)\,.
\end{align*}
We will see in a moment that $\Y_{C/k}^r$ is in fact a quotient space of $\Aut_{C/k}$.

Recall that if $\G$ is an affine algebraic group over a field $k$ and $\H\subset\G$ is a closed subgroup, then a quotient space of ${\G}$ by $\H$ is an algebraic scheme $\X$ over $k$ together with an action $\G\times \X\to \X$ and a rational point $o\in \X(k)$
such that for any $k$-algebra $R$:
\begin{enumerate}[QS1:]
\item $\H(R)=\{g\in \G(R): g\cdot o=o\}$, where $o$ is viewed as element of $\X(R)$, and
\item for any $x\in\X(R)$ there is a faithfully flat $R$-algebra $R'$ and $g'\in \G(R')$ such that $x=g'\cdot o$, where both $x$ and $o$ are viewed as elements of $\X(R')$.
\end{enumerate}
A quotient space always exist for affine groups $\H\subset \G$ as above, and it is unique up to $k$-isomorphism. It is denoted by ${\G}/{\H}$. It is well-known that ${\G}/{\H}$ has dimension $\dim \G-\dim \H$ and that $\G/\H$ is smooth over $k$ if $\G$ is so; see \cite[Ch. III, \S 3]{DemGab70} or \cite[Ch. 3, c.]{Milne17}. 

We return to our analysis of $\Y_{C/k}^r$.
We begin with the particular case in which $C$ is of the form
\begin{align*}
C=k[X_1,\ldots,X_n]/\langle X_1^p,\ldots,X_n^p\rangle=:k[x_1,\ldots,x_n],
\end{align*}
which occurs, for example, when $k$ is algebraically closed. Let $\m=\langle x_1,\ldots,x_n\rangle$ and $B=k[x_{r+1},\ldots,x_n]$. 
Note that $B\in \Y_{C/k}^r(k)$. 
\begin{theorem}\label{th:smoothness over a field}
The natural action of the subgroup $\Aut_{C/k,\{\m\}}\subset\Aut_{C/k}$ on $\Y_{C/k}^r$ induces a $k$-isomorphism
 $${\Aut_{C/k,\{\m\}}}/{\Aut_{C/k,\{B,\m\}}}\cong \Y_{C/k}^r.$$
 Therefore, $\Y_{C/k}^r$ is smooth over $k$ and irreducible of dimension $(p^n-p^{n-r})(n-r)$.
\end{theorem}
\begin{proof}
We will show that $\Y_{C/k}^r$ with the rational point $B\in\Y_{C/k}^r(k)$ verify QS1 and QS2 of the definition of the quotient space ${\Aut_{C/k,\{\m\}}}/{\Aut_{C/k,\{B,\m\}}}$.
Let $R$ be a $k$-algebra. 
By definition, $\Aut_{C/k,\{B,\m\}}(R)=\{\tau\in \Aut_{C/k,\{\m\}}(R): \tau(B\otimes_k R)=B\otimes_k R\}$ so QS1 is trivially satisfied.
Now let $S\in \Y_{C/k}^r(R)$ so that $R\subset S$ and $S\subset C\otimes_k R$ are Galois extensions of differential ranks $n-r$ and $r$ respectively.
We apply Proposition \ref{prop: trivilization} to these extensions, so there exists a finitely presented and faithfully flat $R$-algebra $R'$ and a splitting $p$-basis $z_1,\ldots,z_n$ for $C\otimes_kR\otimes_RR'=C\otimes_k R'$ over $R'$ such that $S\otimes_k R'=R'[z_{r+1},\ldots,z_n]$. By using the decomposition $C\otimes_k R'=R'\oplus (\m\otimes_k R')$ we can even assume that $z_1,\ldots,z_n\in \m\otimes_k R'$, so that $\langle z_1,\ldots,z_n\rangle=\m\otimes_k R'$.
Finally, if by abuse of notation we denote the image of $x_i\in C$ in $C\otimes_k R'$ again by $x_i$, then the assignation $x_i\mapsto z_i$ for each $i=1,\ldots, n$,
induces an $R'$-automorphism of $C\otimes_k R'$ which maps $\m\otimes_k R'$ onto itself and $B\otimes_k R'$ onto $S\otimes_k R'$. This completes the verification of QS2, so the isomorphism of the statement is proved.

Now $\Aut_{C/k,\{\m\}}$ is smooth over $k$ and irreducible by Corollary \ref{cor: about Aut}, so the same holds for $\Y_{C/k}^r$.
In addition, $\dim \Y_{C/k}^r=\dim\Aut_{C/k,\{\m\}}-\dim\Aut_{C/k,\{\m,B\}}=(p^n-p^{n-r})(n-r)$, where the last formula also follows from Corollary \ref{cor: about Aut}. 
\end{proof} 

\medskip

We return to the general situation.
\begin{theorem}\label{th:smoothness over a ring}
Let $A\subset C$ be a Galois extension of exponent one and differential rank $n$. Fix a positive integer $r<n$.
Then the Yuan scheme $\Y_{C/A}^r$ is smooth over $\Spec(A)$ and has smooth irreducible fibers of dimension $(p^n-p^{n-r})(n-r)$. In addition, if $\Spec(A)$ is irreducible (i.e., has a unique generic point), so is $\Y_{C/A}^r$.
\end{theorem}
\begin{proof}
By Proposition \ref{prop: trivilization} there exists a finitely presented faithfully flat $A$-algebra $A'$ such that $C':=C\otimes_AA'\cong A'[X_1,\ldots,X_n]/\langle X_1^p,\ldots,X_n^p\rangle$. Since $\Y_{C'/A'}^r\cong \Y_{C/A}^r\times_A\Spec(A')$, in order to show that $\Y_{C/A}^r$ is smooth over $A$ it is enough to show that $\Y_{C'/A'}^r$ is smooth over $A'$. Now $A'\subset C'$ is a base change of $\mathbb{F}_p\subset \mathbb{F}_p[X_1,\ldots,X_n]/\langle X_1^p,\ldots,X_n^p\rangle$ so the smoothness of $\Y_{C'/A'}$ over $\Spec(A')$ follows from Theorem \ref{th:smoothness over a field}. 

The smoothness of the fibers and the formula for the dimension follow also from Theorem \ref{th:smoothness over a field}. 

We now assume that $\Spec(A)$ is irreducible.
 Since $\Y_{C/A}^r\to\Spec(A)$ is smooth, it is in particular open, and by Theorem \ref{th:smoothness over a field}, the fiber over the generic point of $\Spec(A)$ is irreducible. It follows from \cite[Proposition 3.24]{GorWed10} that $\Y_{C/A}^r$ is irreducible.
\end{proof}

\section*{Acknowledgements}
The first and the third authors were partially supported by PID2022-138916NB-I00, by the Spanish Ministry of Science and Innovation, through the ``Severo Ochoa'' Programme for Centres of Excellence in R\&D (CEX2019-000904-S). The second author was partially supported by CONICET and MINCYT (PICT-2018-02073). The three authors were also partially supported by the Madrid Government under the multiannual Agreement with UAM in the line for the Excellence of the University Research Staff in the context of the V PRICIT (Regional Programme of Research and Technological Innovation) 2022-2024.
The second author would like to thank the hospitality and excellent working conditions at the Department of Mathematics of the UAM, where he carried out part of this work as a Visiting Professor.

\bibliography{References}
\bibliographystyle{abbrv}
\end{document}